\documentclass[a4paper,12pt,leqno]{article}
\usepackage{latexsym}
\usepackage[all]{xy}

\usepackage{amssymb} 
\usepackage{amsmath} 

\usepackage{tikz}
\usetikzlibrary{arrows,cd}  
\usepackage{amsthm}
\usepackage{comment}

\def\P{{\mathbf{P}}}
\def\Z{{\mathbb{Z}}}
\def\K{{\mathbb{K}}}
\def\CC{{\mathbb{C}}}
\def\A{{\mathcal{A}}}
\def\B{{\mathcal{B}}}

\def\O{{\mathcal{O}}}

\DeclareMathOperator{\codim}{codim}

\DeclareMathOperator{\Der}{Der}

\numberwithin{equation}{section}

\newcommand{\owari}{\hfill$\square$}

\newtheorem{theorem}{Theorem}[section]
\newtheorem{prop}[theorem]{Proposition}
\newtheorem{cor}[theorem]{Corollary}
\newtheorem{lemma}[theorem]{Lemma}
\newtheorem{define}[theorem]{Definition}

\theoremstyle{remark}
\newtheorem{rem}[theorem]{Remark}

\title{Equivalence of freeness and higher order freeness of hyperplane arrangements}
\author{Takuro Abe}
\date{\today}

\begin{document}

\maketitle

\begin{abstract}
Using vector bundle theory, we prove that an essential hyperplane
arrangement in an $\ell$-dimensional vector space over a field of
characteristic zero is free if and only if its module of logarithmic
derivations of order $p$ is free for some $1\le p\le\ell-1$.
The same statement holds for multiarrangements.\end{abstract}

\section{Introduction}
Let $\K$ be a field of characteristic zero, $V=\K^\ell$ and 
$S=\K[x_1,\ldots,x_\ell] $ be the coordinate ring of $V$. Let $\Der S=
\oplus_{i=1}^\ell S\partial_{x_i}$ be the $S$-graded module of $\K$-linear derivations of $S$. We say that $\A$ is an \textbf{arrangement of hyperplanes} in $V$ if $\A$ is a finite set of linear hyperplanes in $V$. We say that $\A$ is \textbf{essential} if $\bigcap_{H \in \A} H=\{0\}$. Unless otherwise specified, we assume that an arrangement is essential. For $H \in \A$ fix a linear form $\alpha_H \in V^*$ such that $\ker \alpha_H=H$. Now let us define the key modules in this article.

\begin{define}
Let $D^0(\A):=S$ and for $1 \le p \le \ell$, define 
$$
D^p(\A):=\{\theta \in \wedge^p \Der S \mid 
\theta(\alpha_H,f_2,\ldots,f_p) \in S \alpha_H \ (\forall H \in \A,\ 
\forall f_i \in S)\},
$$
which is called the \textbf{logarithmic derivation module of order $p$}, and 
when $p=1$, it is just called the 
\textbf{logarithmic derivation module} of $\A$.
\label{logder}
\end{define}

Note that $D^p(\A)$ is an $S$-graded module of rank $\begin{pmatrix}
\ell \\ p
\end{pmatrix}$. The case when $p=1$ is intensively studied. In particular, the free arrangement in the following is at the center of the study of them.

\begin{define}
We say that $\A$ is free with $\exp(\A)=(d_1,\ldots,d_\ell)$ if 
$$
D(\A) \simeq \oplus_{i=1}^\ell S[-d_i].
$$
\label{freedef}
\end{define}

One of the important consequences of freeness is Terao's factorization theorem in \cite{T2}, which states that the Poincar\'{e} polynomial of a free arrangement factors according to its exponents.

Recently, the study of $D^p(\A)$ has developed through connections with Solomon–Terao algebras and polynomials \cite{ST, AMMN, CCZZ},  
Hessenberg varieties \cite{AHMMS}, superspace coinvariants and the Fields conjecture \cite{ACKMR, MRW}, and the rational hit conjecture  \cite{Do}. These developments motivate the notion of $p$-freeness.


\begin{define}
Let $0 \le p \le \ell$.
We say that $\A$ is \textbf{$p$-free} if $D^p(\A)$ is free.
\label{pfreedef}
\end{define}

It is trivial by definition that every arrangement is $0$-free and $\ell$-free. The $1$-freeness is nothing but the usual freeness, and the $(\ell-1)$-freeness is the same as 
the usual freeness too due to a canonical identification (Lemma \ref{iden})
$$
D^{\ell-1}(\A)[|\A|] \simeq D(\A)^*.$$
Here for an $S$-graded module $N$, let $N^*:=\mbox{Hom}_S (N,S)$.
Moreover, in \cite{ST}, the following was proved.

\begin{theorem}[Proposition 3.4, \cite{ST}]
If $\A$ is free, then $\wedge^p D(\A) =D^p(\A)$. 
So the freeness implies the $p$-freeness for all $p$.
\label{ST}
\end{theorem}

So it is natural to ask the converse, i.e., whether $\A$ is free if it is $p$-free for some $p$ with $2 \le p \le \ell-2$. This is motivated by Remark 3.10, \cite{DSSWW}. We prove that this is true in general.

\begin{theorem}
Let $\ell \ge 2$ and $\A$ be essential. Then $\A$ is free if and only if there is $p,\ 
1 \le p \le \ell-1$ such that $\A$ is $p$-free.
\label{main}
\end{theorem}

In fact, we prove Theorem \ref{main} in the more general setting of multiarrangements. The proof uses Ballico’s classification of uniform vector bundles of rank $n+1$ on $\P^n$ in \cite{B}.

\medskip

\noindent
\textbf{Acknowledgements}.
The author
is partially supported by JSPS KAKENHI Grant Numbers JP23K17298 and JP25K24692.
\medskip

\noindent
\textbf{AI declaration}.
ChatGPT (OpenAI) was used during the preparation of this article for assistance with mathematical exploration, examples, and language editing.
\medskip

\section{Preliminaries}

We recall the definition of multiarrangements, for which we prove a generalization of Theorem \ref{main}. For background on arrangements and vector bundles, we refer to \cite{OT} and \cite{OSS}, respectively.

\begin{define}
A pair $(\A,m)$ is a \textbf{multiarrangement} if $\A$ is an arrangement, and $m:\A \rightarrow \Z_{>0}$. $m$ is called a \textbf{multiplicity}. Let 
$|m|:=\sum_{H \in \A} m(H)$.
\label{multidef}
\end{define}

We introduce a fundamental combinatorial invariant of $\A$ and an important operation of $(\A,m)$ called a localization.

\begin{define}
For a multiarrangement $(\A,m)$, define 
$$
L(\A):=\{\cap_{H \in \B} H\mid \B \subset \A\},
$$
which is called the \textbf{intersection lattice} of $\A$.
For $X \in L(\A)$
the \textbf{localization} $(\A_X,m_X)$ of $(\A,m)$ at $X$ is defined by 
$$\A_X:=\{
H \in \A \mid H \supset X\}
$$
and by $m_X:=m|_{\A_X}$.
\label{lattice}
\end{define}

For $(\A,m)$, we can also define the logarithmic derivation modules.

\begin{define}
Let $D^0(\A,m):=S$ and for $1 \le p \le \ell$, define 
$$
D^p(\A,m):=\{\theta \in \wedge^p \Der S \mid 
\theta(\alpha_H,f_2,\ldots,f_p) \in S \alpha_H^{m(H)} \ (\forall H \in \A,\ 
\forall f_i \in S)\},
$$
which is called the \textbf{logarithmic derivation module of order $p$} of $(\A,m)$, and 
when $p=1$, it is just called the 
\textbf{logarithmic derivation module} of $(\A,m)$.
\label{logdermulti}
\end{define}

Note that $D^p(\A,m)$ is an $S$-graded module of rank $\begin{pmatrix}
\ell \\ p
\end{pmatrix}$. 
Also, note that $D(\A)=D(\A,1)$ for a constant multiplicity $1$. So the statement for $D(\A,m)$ implies the same statement for $D(\A)$.
We say that $(\A,m)$ is 
\textbf{$p$-free} for $0 \le p \le \ell$ if $D^p(\A,m)$ is free.
Since $D(\A,m)$ is reflexive due to \cite{Z}, the following is immediate.

\begin{prop}
$(\A,m)$ is free if $\ell \le 2$.
\label{2free}
\end{prop}

Also we will use the following.

\begin{prop}[ Corollary A1.13, \cite{Eis} ]
If $M$ is a finitely generated graded reflexive $S$-module, then 
$$
\bigoplus_{d \in \Z} H^0(\widetilde{M}(d))=M.
$$
\label{Eis}
\end{prop}

The following definition and result connect the arrangement with vector bundle theory.

\begin{define}
We say that $(\A,m)$ is \textbf{locally free} if $(\A_X,m_X)$ is free for all 
$
\{0\} \neq X \in L(\A).
$
\label{localfreeness}
\end{define}

\begin{prop}
Let $E:=\widetilde{D(\A,m)}$. Then 
$E$ is locally free if and only if $(\A,m)$ is locally free.
\label{locallyfree}
\end{prop}

\noindent
\textbf{Proof}.
When $m\equiv 1$, then it was proved in Theorem 3.3 in \cite{MS}, and the same proof works well for multiarrangements. \owari
\medskip

Below we gather several results which will be used in the proof of the main result.

\begin{prop}[Proposition 1.6, \cite{ATW}]
For a multiarrangement $(\A,m)$ and $X \in L(\A)$, it holds that 
$$D^p(\A,m)_x = D^p(\A_X,m_X)_x$$
for a prime ideal $x$ with 
$$
x \in X \setminus \cup_{H \in \A \setminus \A_X} H.
$$
\label{localfunctor}
\end{prop}

\begin{prop}
If $(\A,m)$ is $p$-free, then $(\A_X,m_X)$ is $p$-free for all $X \in L(\A)$.
\label{freelocal}
\end{prop}

\noindent
\textbf{Proof}. 
Apply the same proof as the proof of Theorem 4.37 in \cite{OT} combined with Proposition \ref{localfunctor}.\owari
\medskip

\begin{prop}[Lemma 1.3, \cite{ATW}]
If $(\A,m)$ is free, then $$
\wedge^p D(\A,m) =D^p(\A,m)$$
for all $p$.
\label{STmulti}
\end{prop}

\begin{prop}
If $(\A,m)$ is locally free, then 
$$
\wedge^p \widetilde{D(\A,m)} \simeq \widetilde{D^p(\A,m)}
$$
for $ 0 \le p \le \ell$.
\label{localfreesheafisom}
\end{prop}

\noindent
\textbf{Proof}.
Apply the same proof as that of Proposition 3.9 in \cite{MS} combined with Proposition \ref{localfunctor}.\owari
\medskip

\begin{lemma}
Let $(\A,m)$ be a multiarrangement. Then $(\A,m)$ is locally free if and only if 
$(\A_X,m_X)$ is free for $X \in L(\A)$ with $\dim X=1$. 
\label{locfreealmost}
\end{lemma}

\noindent
\textbf{Proof}.
Let $\{0\} \neq Y \in L(\A)$. Then there is $\{0\} \neq X \in L(\A)$ such that 
$X \subset Y$ and $\dim X=1$. It is clear that $\A_X \supset \A_Y$ by definition, and 
$$
(\A_X)_Y=\A_Y
$$
for $Y \in L(\A_X)$.
So Proposition \ref{freelocal} when $p=1$ completes the proof.\owari
\medskip


\begin{lemma}[Lemma 1.4, \cite{ATW}]
Let $X \in L(\A)$ with $\codim X=k$, thus $\A_X \simeq 
\B \times \emptyset_{\ell-k}$ for an essential arrangement $\B$ in $V/X \simeq \K^k$. Then 
$$
D^p(\A_X,m_X) =\bigoplus_{0 \le i \le p }
 D^i(\B,m_X) \otimes_\K D^{p-i}(\emptyset_{\ell-k}).
$$
\label{decomp}
\end{lemma}

\begin{lemma}[Lemma 2.3, \cite{AD}]
It holds that 
$$
D^p(\A,m)[|m|]\simeq D^{\ell-p}(\A,m)^*
$$
for $0 \le p \le \ell$.
\label{iden}
\end{lemma}

Now recall a definition of splitting types and uniform bundles.

\begin{define}
Let $E$ be a vector bundle of rank $r$ on $\P^n_\K$. The 
\textbf{splitting type} $a_E(L)$ of $E$ onto a projective line $L \subset \P^n_\K$ consists of the integers $a_1,\ldots,a_r$ such that 
$$
E|_L \simeq \oplus_{i=1}^r \O_L(a_i).
$$
E is called \textbf{uniform} if $a_E(L)$ is independent of $L$.
\label{uniformdef}
\end{define}

Uniform bundles are intensively studied, and when the rank is at most the dimension of the ambient space, there are several classifications. We use the following result due to Ballico.

\begin{theorem}[Theorem 1, \cite{B}]
Let $E$ be a vector bundle of rank $n+1$ on $\P^n_\K$ for an algebraically closed field $\K$ of characteristic zero. If $n \ge 3$ and $E$ is uniform, then $E$ is a direct sum of line bundles, or line bundle twists of 
$\Omega^1_{\P^n_\K} \oplus \O(a)$ or $
T_{\P^n_\K} \oplus \O(b)$ for some $a, b \in \Z$.
\label{Ballico}
\end{theorem}

\section{Proof of Theorem \ref{main}}

We prove the following as a generalization of Theorem \ref{main}.
Since $\otimes_\K \overline{\K}$ is a faithfully flat functor, the freeness is independent of the field extension. So we may assume that $\K$ is algebraically closed in the rest of this article. 

\begin{theorem}
Let $p$ be an integer with $1 \le p \le \ell-1$ and $(\A,m)$ be a multiarrangement. Then $(\A,m)$ is free if and only if it is $p$-free. 
\label{main2}
\end{theorem}

\noindent
\textbf{Proof}.
The ``only if'' part is nothing but Proposition 
\ref{STmulti}. Let us show the ``if'' part.

Let $(\A,m)$ be $p$-free. We may assume that $2 \le p \le \ell-2$ because of the isomorphism in Lemma \ref{iden}: $$
D^{\ell-1}(\A,m)[|m|] \simeq D(\A,m)^*.
$$
We argue by induction on $\ell \ge 1$. Since we assumed $2 \le p \le \ell-2$, the statement is clear when $\ell \le 3$.

So assume that $\ell \ge 4$. First we show that $(\A,m)$ is locally free.
Let $\{0 \} \neq X \in L(\A)$ with $\codim X=\ell-1$. By Lemma \ref{locfreealmost}, it suffices to show that $(\A_X,m_X)$ is free. Decompose
$$
\A_X=\B \times \emptyset_{1},
$$
where $\B$ is an essential arrangement in $V/X\simeq \K^{\ell-1}$.
By Lemma \ref{decomp}, it holds that 
$$
D^p(\A_X,m_X)=\bigoplus_{i=0}^{p} D^i(\B,m_X) \otimes D^{p-i}(\emptyset_{1})
$$
Since $(\A,m)$ is $p$-free, 
$2 \le p \le \ell-2 $, $\B$ is in $\K^{\ell-1}$ and $D^0(\emptyset_1)=\K[\alpha]$
for some non-zero linear form $\alpha$,
Proposition \ref{freelocal} and the faithful flatness of $\otimes_\K \K[\alpha]$ show that 
$(0) \neq D^p(\B,m_X)$ is free.
Namely, $(\B,m_X)$ is $p$-free.
%
%
So the induction hypothesis shows that $(\B,m_X)$ is free.
Thus Lemma \ref{decomp} and Proposition \ref{STmulti} imply that $(\A_X,m_X)$ is free.
Consequently, $(\A,m)$ is locally free. Thus $E:=\widetilde{D(\A,m)}$ is a vector bundle on 
$\P^{\ell-1} =\mbox{Proj}(S)$ by Proposition \ref{locallyfree}. So 
Proposition \ref{localfreesheafisom} shows that 
$$
\widetilde{\wedge^p D(\A,m) } =
\wedge^p E \simeq \widetilde{D^p(\A,m)}.
$$
Since $(\A,m)$ is $p$-free, we know that $\wedge^p E$ splits into a direct sum of line bundles.

Let $L $ be a projective line in $\P^{\ell-1}=\mbox{Proj}(S)$ and 
let 
$E|_L \simeq \oplus_{i=1}^{\ell} \O_L(a_i).
$
So $a_E(L)=(a_1,\ldots,a_{\ell})$. 
First note that 
$$
s_L:=a_1+\cdots+a_{\ell}=c_1(E)
$$
is constant independent of the choice of the projective line $L$.
We show that in fact 
$$
t_L:=\sum_{i=1}^{\ell} a_i^2
$$
is also independent of $L$. 
Since $\wedge^p E \simeq \widetilde{D^p(\A,m)}$ splits, let us write 
$$
\wedge^p E=\oplus \O(b_i).
$$
Thus 
$$
(\wedge^p E)|_L=\oplus \O_L(b_i) =\wedge^p (E|_L).
$$
and each $b_i$ is $\sum_{i \in I} a_i$, where $I$ is a subset of 
$\{1,\ldots, \ell\}$ with $|I|=p$.
So 
$$
\sum b_i^2=\sum_{I} (\sum_{i \in I} a_i)^2=
\begin{pmatrix}
\ell-1\\p-1
\end{pmatrix}
\sum_{i=1}^\ell a_i^2+
2
\begin{pmatrix}
\ell-2\\p-2
\end{pmatrix}
\sum_{1 \le i<j\le \ell} a_ia_j,
$$
and we can express the right hand side as a linear combination of $s_L^2$ and $t_L$ in such a way that the coefficient of $t_L$ is not zero.
Since the left hand side and $s_L$ are independent of $L$,
 it holds that $t_L$ is also independent of $L$. Let $s=s_L$ and $t=t_L$.

Since $t$ is independent of the choice of projective lines, there is some positive integer $M$ such that, for any projective line $L \subset \P^{\ell-1}$ and the splitting type $a_E(L)=(a_1,\ldots,a_{\ell})$, it holds that 
$$
-M \le a_i \le M
$$
for all $i$.
Note that this $M$ is independent of $L$.
Consider the invariant 
$$
h^0_L:=\sum_{-M \le d \le M }h^0(E|_L(d)).
$$
Since $$
h^0(\O_L(a_i)(d))=\max\{0,a_i+d+1\},
$$
it holds that 
\begin{eqnarray*}
h^0_L&=&\sum_{i=1}^{\ell} \displaystyle \frac{(a_i+M+1)(a_i+M+2)}{2}
=\sum_{i=1}^{\ell}\displaystyle \frac{a_i^2+(2M+3) a_i+(M+1)(M+2)}{2}
\\
&=&\displaystyle \frac{t+s(2M+3)+\ell(M+1)(M+2)}{2}.
\end{eqnarray*}
Namely, $h^0_L$ is independent of $L$ too.

On the other hand, there is a non-empty open set $U$ in the Grassmannian of lines in $\P^{\ell-1}$ such that $a_E(L) $ is constant on  $L \in U$, such a splitting type is called a generic splitting type 
(see Chapter I, Lemma 3.2.2 in \cite{OSS} for example). By the upper semi-continuity of $L \mapsto h^0(E|_L(d))$, we know that for $L \in U$, and $H \not \in U$, it holds that 
$$
h^0(E(d)|_H) \ge h^0(E(d)|_L) 
$$
for all $-M \le d \le M$. Since their sums $h^0_H$ and $h^0_L$ are equal, we know that 
$$
h^0(E(d)|_H) = h^0(E(d)|_L) 
$$
for all $-M \le d \le M$ and all $H,L$. Then we can read off the splitting type of $E$ onto any line $L$. Namely, for a splitting type 
$a_E(L)=(c_1,\ldots,c_\ell)$, it is easy to show that 
$$
h^0(E|_L(d))-h^0(E|_L(d-1))=|\{i\mid 
1 \le i \le \ell,\ c_i \ge -d\}|.
$$
Thus 
\begin{eqnarray*}
&\ &\{h^0(E|_L(d))-h^0(E|_L(d-1))\}-
\{h^0(E|_L(d-1))-h^0(E|_L(d-2))\}\\
&=&|\{i\mid 
1 \le i \le \ell,\ c_i = -d\}|.
\end{eqnarray*}
Since these dimensions are independent of $L$ for each $d$,
it shows that every splitting type is the same as 
the generic splitting type. 

So $E$ is a uniform bundle of rank $\ell$ on $\P^{\ell-1}$. Since $\wedge^p E$ splits by assumption but no $p$-th wedge products for $1 \le p \le \ell-1$ of non-splitting bundles appearing in Theorem \ref{Ballico} split, Theorem \ref{Ballico} shows that $E$ splits into a direct sum of line bundles when $\ell \ge 4$. Now since $D(\A,m)$ is reflexive, Proposition \ref{Eis} shows that 
$$
\bigoplus_{d \in \Z} H^0(\widetilde{D(\A,m)}(d))=D(\A,m),
$$
which completes the proof.
\owari
\medskip

\begin{rem}
The key point above is that a vector bundle $E$ of rank $n+1$ on $\P^n$ splits if $\bigwedge^p E$ splits for some $1\le p\le n$. There may be alternative proofs of this implication using classical vector bundle theory. In fact, by using the same argument as the proof of Theorem \ref{main2} combined with Theorems \ref{Ballico} and \ref{EHS} below, we can show the following.
\end{rem}

\begin{theorem}[\cite{EHS}]
Let $E$ be a rank $r$-locally free sheaf on $\P^n_\CC$ with $r \le n$. If $E$ is uniform, then either $E$ splits, $E \simeq T_{\P^n}(a)$, 
or $E \simeq \Omega_{\P^n}(b)$ for some $a,b \in \Z.$
\label{EHS}
\end{theorem}

\begin{theorem}
Let $n \ge 3$, $E$ be a rank $r$-locally free sheaf on $\P^n_\CC$ with $r \le n+1$. If $\wedge^p E$ splits for some $p,\ 1 \le p \le r-1$, then $E$ splits.
\label{main3}
\end{theorem}

\noindent
\textbf{Proof.}
Apply the same argument as the proof of Theorem \ref{main2} combined with Theorems \ref{Ballico} and \ref{EHS}.\owari

\end{document}